\documentclass[12pt,a4paper]{article}
\usepackage{amssymb,amsmath,amsthm}
\usepackage{graphicx}

\newcommand\cC{{\mathcal C}}

\newcommand\cF{{\mathcal F}}

\newcommand{\abs}[1]{\left\lvert{#1}\right\rvert}

\newcommand\bC{\mathbf C}
\newcommand\cG{{\mathcal G}}

\newcommand\cQ{{\mathcal Q}}

\theoremstyle{plain}
\newtheorem{theorem}{Theorem}[section]

\newtheorem{conjecture}[theorem]{Conjecture}

\theoremstyle{definition}

\newtheorem{claim}[theorem]{Claim}

\newtheorem*{remark}{Remark}

\newcommand\tref[1]{Theorem~\ref{thm:#1}}
\newcommand\cref[1]{Corollary~\ref{cor:#1}}

\title{Forbidding rank-preserving copies of a poset}
\author{D\'aniel Gerbner$^{a,}$\thanks{Research supported by the J\'anos Bolyai Research Fellowship of the Hungarian Academy of Sciences and the National Research, Development and Innovation Office -- NKFIH under the grant K 116769.}, Abhishek Methuku$^b$, D\'aniel T. Nagy$^{a,}$\thanks{Research supported by the \'{U}NKP-17-3 New National Excellence Program of the Ministry of Human Capacities and by National Research, Development and Innovation Office - NKFIH under the grant K 116769.}, \\ Bal\'azs Patk\'os$^{a,}$\thanks{Research supported by the National Research, Development and Innovation Office -- NKFIH under the grants SNN 116095 and K 116769.}, M\'at\'e Vizer$^{a,}$\thanks{Research supported by the National Research, Development and Innovation Office -- NKFIH under the grant SNN 116095.} \\
\small $^a$ Alfr\'ed R\'enyi Institute of Mathematics, Hungarian Academy of Sciences\\
\small P.O.B. 127, Budapest H-1364, Hungary.\\
\small $^b$ Central European University, Department of Mathematics\\
\small Budapest, H-1051, N\'ador utca 9.\\
\medskip
\small \texttt{\{gerbner,nagydani,patkos\}@renyi.hu, \{abhishekmethuku,vizermate\}@gmail.com}
\medskip}

\begin{document}
\maketitle

\begin{abstract}

The maximum size, $La(n,P)$, of a family of subsets of $[n]=\{1,2,...,n\}$ without containing a copy of $P$ as a subposet, has been intensively studied. 

Let $P$ be a graded poset. We say that a family $\cF$ of subsets of $[n]=\{1,2,...,n\}$ contains a \emph{rank-preserving} copy of $P$ if it contains a copy of $P$ such that elements of $P$ having the same rank are mapped to sets of same size in $\cF$. The largest size of a family of subsets of $[n]=\{1,2,...,n\}$  without containing a rank-preserving copy of $P$ as a subposet is denoted by $La_{rp}(n,P)$. Clearly, $La(n,P) \le La_{rp}(n,P)$ holds.

In this paper we prove asymptotically optimal upper bounds on $La_{rp}(n,P)$ for tree posets of height $2$ and monotone tree posets of height $3$, strengthening a result of Bukh in these cases. We also obtain the exact value of $La_{rp}(n,\{Y_{h,s},Y_{h,s}'\})$ and $La(n,\{Y_{h,s},Y_{h,s}'\})$, where  $Y_{h,s}$ denotes the poset on $h+s$ elements $x_1,\dots,x_h,y_1,\dots,y_s$ with $x_1<\dots<x_h<y_1,\dots,y_s$ and $Y'_{h,s}$ denotes the dual poset of $Y_{h,s}$.

\end{abstract}

\section{Introduction}
In extremal set theory, many of the problems considered can be phrased in the following way: what is the size of the largest family of sets that satisfy a certain property. The very first such result is due to Sperner \cite{sperner} which states that if $\cF$ is a family of subsets of $[n]=\{1,2\dots,n\}$ (we write $\cF\subseteq 2^{[n]}$ to denote this fact) such that no pair $F,F'\in \cF$ of sets are in inclusion $F\subsetneq F'$, then $\cF$ can contain at most $\binom{n}{\lfloor n/2\rfloor}$ sets. This is sharp as shown by $\binom{[n]}{\lfloor n/2\rfloor}$ (the family of all $k$-element subsets of a set $X$ is denoted by $\binom{X}{k}$ and is called the $k^{th}$ \textit{layer} of $X$). If $P$ is a poset, we denote by $\le_P$ the partial order acting on the elements of $P$. Generalizing Sperner's result, Katona and Tarj\'an \cite{KatTar} introduced the problem of determining the size of the largest family $\cF\subseteq 2^{[n]}$ that does not contain sets satisfying some inclusion patterns. Formally, if $P$ is a finite poset, then a subfamily $\cG\subseteq \cF$ is
\begin{itemize}
\item
a \textit{(weak) copy} of $P$ if there exists a bijection $\phi: P \rightarrow \cG$ such that we have $\phi(x)\subseteq \phi(y)$ whenever $x \le_P y$ holds,
\item
a \textit{strong} or \textit{induced copy} of $P$ if there exists a bijection $\phi: P \rightarrow \cG$ such that we have $\phi(x)\subseteq \phi(y)$ if and only if $x \le_P y$ holds.
\end{itemize}
A family is said to be \textit{$P$-free} if it does not contain any (weak) copy of $P$ and \textit{induced $P$-free} if it does not contain any induced copy of $P$. Katona and Tarj\'an started the investigation of determining
\[
La(n,P):=\max\{|\cF|:\cF\subseteq 2^{[n]}, ~\text{$\cF$ is $P$-free}\} 
\]
and
\[ 
La^*(n,P):=\max\{|\cF|:\cF\subseteq 2^{[n]}, ~\text{$\cF$ is induced $P$-free}\}.
\]
The above quantities have been determined precisely or asymptotically for many classes of posets (see \cite{griggsli} for a nice survey), but the question has not been settled in general. Recently, Methuku and P\'alv\"olgyi \cite{MetPal} showed that for any poset $P$, there exists a constant $C_P$ such that $La(n,P)\le La^*(n,P)\le C_P\binom{n}{\lfloor n/2\rfloor}$ holds (the inequality $La(n,P)\le |P|\binom{n}{\lfloor n/2\rfloor}$ follows trivially from a result of Erd\H os \cite{erdos}). However, it is still unknown whether the limits $\pi(P)=\lim_{n \rightarrow \infty}\frac{La(n,P)}{\binom{n}{\lfloor n/2\rfloor}}$ and $\pi^*(P)=\lim_{n \rightarrow \infty}\frac{La^*(n,P)}{\binom{n}{\lfloor n/2\rfloor}}$ exist. In all known cases, the asymptotics of $La(n,P)$ and $La^*(n,P)$ were given by ``taking as many middle layers as possible without creating an (induced) copy of $P$". Therefore researchers of the area believe the following conjecture that appeared first in print in \cite{griggslu}.

\begin{conjecture}\label{conj}
(i) For any poset $P$ let $e(P)$ denote the largest integer $k$ such that for any $j$ and $n$ the family $\cup_{i=1}^k\binom{[n]}{j+i}$ is $P$-free. Then $\pi(P)$ exists and is equal to $e(P)$.

(ii) For any poset $P$ let $e^*(P)$ denote the largest integer $k$ such that for any $j$ and $n$ the family $\cup_{i=1}^k\binom{[n]}{j+i}$ is induced $P$-free. Then $\pi^*(P)$ exists and is equal to $e^*(P)$.
\end{conjecture}

Let $P$ be a graded poset with rank function $\rho$. Given a family $\cF$, a subfamily $\cG\subseteq \cF$ is a rank-preserving copy of $P$ if $\cG$ is a (weak) copy of $P$ such that elements having the same rank in $P$ are mapped to sets of same size in $\cG$. More formally, $\cG\subseteq \cF$ is a rank-preserving copy of $P$ if there is a bijection $\phi: P \rightarrow \cG$ such that $|\phi(x)| = |\phi(y)|$ whenever $\rho(x) = \rho(y)$ and we have $\phi(x)\subseteq \phi(y)$ whenever $x \le_P y$ holds. A family $\cF$ is rank-preserving $P$-free if it does not contain a rank-preserving copy of $P$. In this paper, we study the function
\[
La_{rp}(n,P):=\max\{|\cF|:\cF\subseteq 2^{[n]}, ~\text{$\cF$ is rank-preserving $P$-free}\}. 
\]

In fact, our problem is a natural special case of the following general problem introduced by Nagy \cite{Nagy}. Let $c: P\rightarrow [k]$ be a coloring of the poset $P$ such that for any $x\in [k]$ the pre-image $c^{-1}(x)$ is an antichain. A subfamily $\cG\subseteq \cF$ is called a \textit{$c$-colored copy} of $P$ in $\cF$ if $\cG$ is a (weak) copy of $P$ and sets corresponding to elements of $P$ of the same color have the same size. Nagy investigated the size of the largest family $\cF\subseteq 2^{[n]}$ which does not contain a $c$-colored copy of $P$, for several posets $P$ and colorings $c$. Note that when $c$ is the rank function of $P$, then this is equal to $La_{rp}(n,P)$. Nagy also showed that there is a constant $C_P$ such that $La_{rp}(n,P) \le C_P \binom{n}{\lfloor n/2\rfloor}.$

A complete multi-level poset is a poset in which every element of a level is related to every element of another level. Note that any rank-preserving copy of a complete multi-level poset $P$ is also an induced copy of $P$. In fact, in \cite{Patkos}, Patk\'os determined the asymptotics of $La^*(n,P)$,  for some complete multi-level posets $P$ by finding a rank preserving copy of $P$. 

By definition, for every graded poset $P$ we have $La(n,P) \le La_{rp}(n,P)$. 
Boehnlein and Jiang \cite{boehnlein2012set} gave a family of posets $P$ showing that the difference between $La^*(n, P)$ and $La(n, P)$ can be arbitrarily large. Since their posets embed into a complete multi-level poset of height $3$ in a rank-preserving manner, the above mentioned result of Patk\'os implies that for the same family of posets, $La_{rp}(n, P)$ can be arbitrarily smaller than $La^*(n, P)$. However, it would be interesting to determine if the opposite phenomenon can occur.


\subsection{Our results}

\subsubsection*{Asymptotic results}
For a poset $P$ its \textit{Hasse diagram}, denoted
by $H(P)$, is a graph whose vertices are elements of $P$, and $xy$ is an edge if $x < y$ and there is no other element $z$ of $P$ with $x < z < y$. We call a poset, \textit{tree poset} if $H(P)$ is a tree. A tree poset is called \emph{monotone increasing} if it has a unique minimal element and it is called \emph{monotone decreasing} if it has a unique maximal element. A tree poset is \emph{monotone} if it is either monotone increasing or decreasing.


A remarkable result concerning Conjecture \ref{conj} is that of Bukh \cite{Bukh}, who verified Conjecture \ref{conj} (i) for tree posets. In the following results we strengthen his result in two cases.

\begin{theorem}\label{thm:height2}
Let $T$ be any tree poset of height $2$. Then we have $$La_{rp}(n,T)= \left(1 + O_T\left (\sqrt{\frac{\log n}{n}}\right)\right) \binom{n}{\lfloor \frac{n}{2} \rfloor}.$$
\end{theorem}

\begin{theorem}\label{thm:height3}
Let $T$ be any monotone tree poset of height 3. Then we have $$La_{rp}(n,T)=\left(2+O_T\left(\sqrt{\frac{\log n}{n}}\right)\right)\binom{n}{\lfloor \frac{n}{2}\rfloor}.$$
\end{theorem}

The lower bounds in Theorem \ref{thm:height2} and Theorem \ref{thm:height3} follow simply by taking one and two middle layers of the Boolean lattice of order $n$, respectively.

\subsubsection*{An exact result}

The \textit{dual} of a poset $P$ is the poset $P'$ on the same set with the partial order relation of $P$ replaced by its inverse,  i.e., $x \le y$ holds in $P$ if and only if $y \le x$ holds in $P'$. Let $Y_{h,s}$ denote the poset on $h+s$ elements $x_1,\dots,x_h,y_1,\dots,y_s$ with $x_1<\dots<x_h<y_1,\dots,y_s$ and let $Y'_{h,s}$ denote the dual of $Y_{h,s}$. Let $\Sigma(n,h)$ for the number of elements on the $h$ middle layers of the Boolean lattice of order $n$, so $\Sigma(n,h)= \sum_{i=1}^{h} \binom{n}{\lfloor \frac{n-h}{2} \rfloor +i}$.

Investigation on $La(n,Y_{h,s})$ was started by Thanh in \cite{T1998}, where asymptotic results were obtained. Thanh also gave a construction showing that $La(n,Y_{h,s}) > \Sigma(n,h)$, from which it easily follows that $La(n,Y'_{h,s}) > \Sigma(n,h)$ as well. Interestingly,  De Bonis and Katona \cite{DKS2005} showed that if both $Y_{2,2}$ and $Y'_{2,2}$ are forbidden, then an exact result can be obtained: $La(n,\{Y_{2,2},Y'_{2,2}\})=\Sigma(n,2)$. Later this was extended by Methuku and Tompkins \cite{MT2015}, who proved $La(n,\{Y_{k,2},Y'_{k,2}\})=\Sigma(n,k)$, and   $La^*(n,\{Y_{2,2},Y'_{2,2}\})=\Sigma(n,2)$. Very recently, Martin, Methuku, Uzzell and Walker \cite{MMSUW2017} and independently, Tompkins and Wang \cite{TompkinsWang} showed that $La^*(n,\{Y_{k,2},Y'_{k,2}\})=\Sigma(n,k)$. We prove the following theorem which extends all of these previous results and proves a conjecture of \cite{MMSUW2017}.


\begin{theorem}\label{thm:reversed}
For any pair $s,h\ge 2$ of positive integers, there exists $n_0=n_0(h,s)$ such that for any $n\ge n_0$ we have $$La_{rp}(n,\{Y_{h,s},Y'_{h,s}\})=\Sigma(n,h).$$
\end{theorem}

The lower bound trivially follows by taking $h$ middle layers of the Boolean lattice of order $n$. (Note that adding any extra set creates a rank-preserving copy of either $Y_{h,s}$ or $Y'_{h,s}$.) Moreover, any rank-preserving copy of $Y_{h,s}$ (respectively $Y'_{h,s}$) is also an induced copy of $Y_{h,s}$ (respectively $Y'_{h,s}$). Therefore, Theorem \ref{thm:reversed} implies that $La^*(n,\{Y_{h,s},Y'_{h,s}\}) = La(n,\{Y_{h,s},Y'_{h,s}\}) = \Sigma(n,h)$.

\begin{remark}
One wonders if the condition $h\ge 2$ is necessary in Theorem \ref{thm:reversed}. Katona and Tarj\'an \cite{KatTar} proved that $La(n, \{Y_{1,2},Y'_{1,2}\})=\binom{n}{n/2}$ if $n$ is even and 
$La(n, \{Y_{1,2},Y'_{1,2}\})=2\binom{n-1}{(n-1)/2}>\binom{n}{n/2}$ if $n$ is odd. The following construction shows that no matter how little we weaken the condition of being $\{Y_{1,2},Y'_{1,2}\}$-free, there are families strictly larger than $\binom{n}{n/2}$ even in the case $n$ is even. Let us define
\[
\cF_{2,3}=\left\{F\in \binom{[n]}{n/2+1}:n-1,n\in F\right\}\cup\left\{F\in \binom{[n]}{n/2}: |F\cap \{n-1,n\}|\le 1\right\}.
\]
Observe that $\cF_{2,3}$ is $\{Y_{1,2},Y'_{1,3}\}$-free and its size is $\binom{n-2}{n/2+1}+(\binom{n}{n/2}-\binom{n-2}{n/2-2})>\binom{n}{n/2}$.

\end{remark}

\section{Proofs} 
Using Chernoff's inequality, it is easy to show (see for example \cite{griggslu}) that the number of sets $F \subset [n]$ of size more than $n/2 + 2\sqrt{n\log n}$ or smaller than $n/2 - 2 \sqrt{n \log n}$ is at most 
\begin{equation}
\label{eq:reduction}
O\left(\frac{1}{n^{3/2}}\binom{n}{n/2}\right).
\end{equation}
Thus in order to prove \tref{height2} and \tref{height3}, we can assume the family only contains sets of size more than $n/2 - 2\sqrt{n\log n}$ and smaller than $n/2 + 2 \sqrt{n \log n}$.

\subsection{Proof of \tref{height2}: Trees of height two}
The proof of \tref{height2} follows the lines of a reasoning of Bukh's \cite{Bukh}. The new idea is that we count the number of related pairs between two fixed levels as detailed in the proof below.

Let $\mathcal F$ be a $T$-free family of subsets of $[n]$ and let the number of elements in $T$ be $t$. Using \eqref{eq:reduction}, we can assume $\mathcal F$
only contains sets of sizes in the range $[n/2 - 2\sqrt{n\log n}, n/2 + 2\sqrt{n\log n}]$. A pair of sets $A, B \in \cF$ with $A \subset B$ is called a $2$-chain in $\cF$. It is known by a result of Kleitman \cite{Kleitman} that the number of $2$-chains in $\mathcal F$ is at least 
\begin{equation}
\label{lower}
\left (\abs{\cF} - \binom{n}{\lfloor \frac{n}{2} \rfloor} \right)\frac{n}{2}.
\end{equation}

For any $n/2 - 2\sqrt{n\log n} \le i \le n/2 + 2\sqrt{n\log n}$, let $\cF_i := \cF \cap \binom{[n]}{i}$. 
\begin{claim}
\label{2chains}
For any $i < j$, the number of $2$-chains $A \subset B$ with $A \in \cF_i$ and $B \in \cF_j$ is at most $(t-2)(\abs{\cF_i}+ \abs{\cF_j})$.
\end{claim}
\begin{proof}
Suppose otherwise, and construct an auxiliary graph $G$ whose vertices are elements of $\cF_i$ and $\cF_j$, and two vertices form an edge of $G$ if the corresponding elements form a $2$-chain. This implies that $G$ contains more than $(t-2)(\abs{\cF_i}+ \abs{\cF_j})$ edges, so it has average degree more than $2(t-2)$. One can easily find a subgraph $G'$ of $G$ with minimum degree at least $t-1$, into which we can greedily embed any tree with $t$ vertices. So in particular, we can find $T$ in $G'$ which corresponds to a rank-preserving copy of $T$ into $\cF$, a contradiction.
\end{proof}

Claim \ref{2chains} implies that the total number of $2$-chains in $\cF$ is at most $$\sum_{n/2 - 2\sqrt{n\log n} \le i < j \le n/2 + 2\sqrt{n\log n}} (t-2)(\abs{\cF_i}+ \abs{\cF_j}) = (t-2)(4\sqrt{n\log n})\abs{\cF}.$$

Combining this with \eqref{lower}, and simplifying we get 
$$ \abs{\cF} \left(1 - 8(t-2)\sqrt{\frac{\log n}{n}}\right) \le \binom{n}{\lfloor \frac{n}{2} \rfloor}.$$
Rearranging, we get
$$ \abs{\cF}  \le \binom{n}{\lfloor \frac{n}{2} \rfloor} \left(1 + O_T\left (\sqrt{\frac{\log n}{n}}\right)\right)$$
as desired. 

\qed

\subsection{Proof of  \tref{height3}: Monotone trees of height three}

First note that it is enough to prove the statement for $T=T_{r,3}$ the monotone increasing tree poset of height tree where all elements, except its leaves (i.e., its elements on the top level) have degree $r$. Let $\cF\subseteq 2^{[n]}$ be a family of sets which does not contain any rank-preserving copies of $T_{r,3}$. Using \eqref{eq:reduction} we can assume that for any set $F\in \cF$ we have $|F-n/2|\le 2 \sqrt{n \log n}$.


We will prove that for such a family, \begin{equation}\label{bumm}
\sum_{F\in \cF}|F|!(n-|F|)!\le (2+O_r(1/n))n!
\end{equation} holds. This is enough as dividing by $n!$ yields
\[
\frac{|\cF|}{\binom{n}{\lfloor n/2\rfloor}}\le \sum_{F\in \cF}\frac{1}{\binom{n}{|F|}}\le (2+O_r(1/n))
\]
and hence the statement of the theorem will follow.

Observe that $\sum_{F\in \cF}|F|!(n-|F|)!$ is the number of pairs $(F,\cC)$ where $F\in \cF\cap \cC$ and $\cC$ is a maximal chain in $[n]$. We will use the chain partitioning method introduced in \cite{partition}. For any $G\in \cF$ we define $\mathbf{C}_G$ to be the set of maximal chains $\cC$ in $[n]$ such that the smallest set of $\cC\cap \cF$ is $G$. 

To prove (\ref{bumm}) it is enough to show that for any fixed $G\in \cF$ the number of pairs $(F,\cC)$ with $F\in \cF\cap \cC$, $\cC\in \bC_G$ is at most $(2+O_r(1/n))|\bC_G|$. We count the number of these pairs $(F,\cC)$ in three parts. 

Firstly, the number of pairs where either $F=G$ or $F$ is the second smallest element of $\cF\cap \cC$ is at most $2|\bC_G|$ (there might be chains in $\bC_G$ with $\cC\cap \cF=\{G\}$).

Let us consider the following sub-partition of $\bC_G$. For any $G\subsetneq G'\in \cF$ let $\bC_{G,G'}$ denote the set of maximal chains $\cC$ such that $G$ and $G'$ are the smallest and second smallest sets in $\cF\cap \cC$, respectively. Observe that $|\bC_{G,G'}|=m_G\cdot m_{G,G'}\cdot (n-|G'|)!$, where $m_G$ is the number of chains from $\emptyset$ to $G$ that do not contain any other sets from $\cF$ and $m_{G,G'}$ is the number of chains from $G$ to $G'$ that do not contain any other sets from $\cF$.

Secondly, let us now count the pairs $(F,\cC)$ such that $F\in \cF\cap \cC$, $\cC\in \bC_{G,G'}$ and there are less than $r^2$ sets $F'\in \cF$ with $|F'|=|F|$, $G'\subsetneq F'$. To this end, let us fix $G'$ and count such pairs $(F,\cC)$. All sets in $\cF$ have size at most $n/2+2\sqrt{n\log n}$ and at least $n/2-2\sqrt{n\log n}$, so $|G'|\ge n/2-2\sqrt{n\log n}$. For a set $F\supsetneq G'$ the number of chains in $\bC_{G,G'}$ that contain $F$ is $m_Gm_{G,G'}\cdot (|F|-|G'|)!(n-|F|)!$, thus we obtain that the number of such pairs is at most
\[
\sum_{i=1}^{4\sqrt{n\log n}}r^2m_Gm_{G,G'}\cdot i!(n-|G'|-i)!\le 2r^2m_Gm_{G,G'}(n-|G'|-1)!=\frac{2r^2}{n-|G'|}|\bC_{G,G'}|\le \frac{5r^2}{n}|\bC_{G,G'}|.
\]
Summing this for all $G'$ we obtain that the total number of such pairs $(F,\cC)$ of this second type is at most $\frac{5r^2}{n}|\bC_{G}|$.

Finally, let us count the pairs $(F, \cC)$ with $F\in \cC\cap \cF$, $\cC\in \bC_{G,G'}$ and there are at least $r^2$ many sets $F'\in\cF$ with $G'\subsetneq F'$, $|F'|=|F|$. To this end we group some of the $\bC_{G,G'}$'s together. Let $$\bC_{G,k}:=\cup_{G': |G'|=k}\bC_{G,G'},\hskip 0.3truecm \cF_{G,k}:=\{G'\in \cF:G\subseteq G', |G'|=k\}$$ and let us introduce the function $f_{G,k}:\cF_{G,k} \rightarrow [n]$ by $$f_{G,k}(G'):=\{j: \exists F_1,F_2,\dots,F_{r^2}, \,\textnormal{such that}\, G'\subseteq F_i, |F_i|=j ~\text{for all}\ i=1,2,\dots, r^2\}.$$
Observe that for any distinct $G_1',G_2',\dots, G_r'\in \cF_{G,k}$ we have $\cap_{i=1}^kf_{G,k}(G_i')=\emptyset$. Indeed, if $j\in \cap_{i=1}^kf_{G,k}(G_i')=\emptyset$, then one could extend $G,G_1',G_2',\dots, G_r'$ to a rank-preserving copy of $T_{r,3}$ such that all sets corresponding to leaves of $T_{r,3}$ are of size $j$.

Note that by the assumption on the set sizes of $\cF$, the function $f_{G,k}$ maps to $[n/2-2\sqrt{n\log n},n/2+2\sqrt{n\log n}]$, so its range has size at most $4 \sqrt{n\log n}$. As every maximal chain contains exactly one set of size $j$ (not necessarily contained in $\cF$), we obtain that the number of pairs $(F,\cC)$ with $F\in \cF\cap \cC$, $\cC\in \bC_{G,k}$ is at most 
\begin{equation}\label{bummbumm}
m_G\cdot 4\sqrt{n\log n}(r-1)(k-|G|)!(n-k)!.
\end{equation}
Indeed, if the size $j$ of $F$ is fixed, then $j$ belongs to $f_{G,k}(G')$ for at most $r-1$ sets $G'\in \cF_{G,k}$, so for this particular $j$ the number of pairs is at most $m_G\cdot (r-1)(k-|G|)!(n-k)!$.

Summing up (\ref{bummbumm}) for all $k>|G|$ we obtain that the number of pairs $(F,\cC)$ of this third type is at most 
\begin{equation*}
\begin{split}
\sum_{k=|G|+1}^{n/2+2\sqrt{n\log n}}m_G\cdot 4\sqrt{n\log n}(r-1)(k-|G|)!(n-k)! & \le \frac{8(r-1)\sqrt{n\log n}}{n-|G|}m_G(n-|G|)!\\
& \le \frac{17(r-1) \sqrt{n\log n}}{n}|\bC_G|.
\end{split}
\end{equation*}
 Adding up the estimates on the number of pairs $(F,\cC)$ of these 3 types, completes the proof.
\qed

\subsection{Proof of \tref{reversed}: $\{Y_{h,s}, Y'_{h,s}\}$-free families}

Let $\cF\subset 2^{[n]}$ be a family not containing a rank-preserving copy of $Y_{h,s}$ or $Y'_{h,s}$. First, we will introduce a weight function. For every $F\in\cF$, let $w(F)=\binom{n}{|F|}$. For a maximal chain $\cC$, let $w(\cC)=\sum_{F\in \cC\cap \cF}w(F)$ denote the weight of $\cC$. Let $\bC_n$ denote the set of maximal chains in $[n]$. Then
$$\frac{1}{n!}\sum_{\cC\in\bC_n} w(\cC)=\frac{1}{n!}\sum_{\cC\in\bC_n}\sum_{F\in \cC\cap \cF}w(F)=\frac{1}{n!}\sum_{F\in \cF} |F|!(n-|F|)!w(F)=|\cF|.$$

This means that the average of the weight of the full chains equals the size of $\cF$. Therefore it is enough to find an upper bound on this average. We will partition $\bC_n$ into some parts and show that the average weight of the chains is at most $\Sigma(n,h)$ in each of the parts. Therefore this average is also at most $\Sigma(n,h)$, when calculated over all maximal chains, which gives us $|\cF|\le \Sigma(n,h)$.

Let $\cG=\{F\in\cF~|~\exists P,Q\in\cF\backslash\{F\},~P\subset F\subset Q\}$. Let $A_1\subset A_2\subset\dots \subset A_{h-1}$ be $h-1$ different sets of $\cG$. Then we define $\bC(A_1, A_2,\dots ,A_{h-1})$ as the set of those chains that contain all of ${A_1, A_2,\dots A_{h-1}}$ and these are the $h-1$ smallest elements of $\cG$ in them. We also define $\bC_-$ as the set of those chains that contain at most $h-2$ elements of $\cG$. Then the sets of the form $\bC(A_1, A_2,\dots A_{h-1})$ together with $\bC_-$ are pairwise disjoint and their union is $\bC_n$.

Now we will show the average weight within each of these sets of chains is at most $\Sigma(n,h)$. This is easy to see for $\bC_-$. If $\cC\in \bC_-$, then $|\cC\cap\cF|\le h$, since every element of $\cF\cap \cC$ except for the smallest and the greatest must be in $\cG$. Therefore $c(W)\le\Sigma(n,h)$ for every $\cC\in \bC_-$, which trivially implies
$$\sum_{\cC\in \bC_-}w(\cC)\le |\bC_-|\Sigma(n,h).$$

Now consider some sets $A_1\subset A_2\subset\dots \subset A_{h-1}$ in $\cG$ such that $\bC(A_1, A_2,\dots A_{h-1})$ is non-empty. We will use the notations $\bC(A_1, A_2,\dots ,A_{h-1})=\cQ$, $|A_1|=\ell_1$ and $n-|A_{h-1}|=\ell_2$  for simplicity. Note that the chains in $\cQ$ do not contain any member of $\cF$ of size between $|A_1|$ and $|A_{h-1}|$ other than the sets $A_2, A_3 \dots A_{h-2}$. Such a set would be in $\cG$ (since it contains $A_1$ and is contained in $A_{h-1}$), therefore its existence would contradict the minimality of $\{A_1, A_2, \dots ,A_{h-1}\}$. The chains in $\cQ$ must also avoid all subsets of $A_1$ that are in $\cG$ for the same reason.

Let $N_1$ denote the number of chains between $\emptyset$ and $A_1$ that avoid the elements of $\cG$ (except for $A_1$). Let $N_2$ denote the number of chains between $A_1$ and $A_{h-1}$ that contain the sets $A_2, A_3, \dots ,A_{h-2}$, but no other element of $\cF$. Then $|\cQ|=N_1 N_2 \ell_2!$.

Now we will investigate how much the sets of certain sizes can contribute to the sum
\begin{equation}\label{somesumhere}
\sum_{\cC\in \cQ}w(\cC).
\end{equation}

The sets $A_1, A_2, \dots A_{h-1}$ appear in all chains of $\cQ$, so their contribution to the sum is
$$|\cQ|\sum_{i=1}^{h-1} w(A_i)=|\cQ|\sum_{i=1}^{h-1} \binom{n}{|A_i|}\le|\cQ|\Sigma(n,h-1).$$

We have already seen that there are no other sets of $\cF$ in these chains with a size between $|A_1|$ and $|A_{h-1}|$.

If $\ell_1<\frac{n}{2}-2\sqrt{n\log n}$, then (by \eqref{eq:reduction}) the contribution coming from the subsets of $A_1$ is trivially at most $$|\cQ|\sum_{i=0}^{\ell_1-1}\binom{n}{i}=|\cQ|O\left(\binom{n}{n/2}\frac{1}{n^{3/2}}\right).$$
The contribution coming from supersets of $A_{h-1}$ is similarly small if $\ell_2<\frac{n}{2}-2\sqrt{n\log n}$.
From now on we consider the cases when $\ell_1\ge \frac{n}{2}-2\sqrt{n\log n}$ and $\ell_2\ge \frac{n}{2}-2\sqrt{n\log n}$.

There are no $s$ supersets of $A_{h-1}$ of equal size in $\cF$, since these would form a rank-preserving copy of $Y_{h,s}$ together with the sets $A_1, A_2, \dots A_{h-1}$ and some set $P\in\cF$, $P\subset A_1$. (Such a set exists, since $A_1\in\cG$.)

A superset of $A_{h-1}$ of size $n-i$ appears in $|\cQ|\binom{\ell_2}{i}^{-1}$ chains of $\cQ$. Therefore the total contribution to the sum (\ref{somesumhere}) by supersets of $A_{h-1}$ is at most
\begin{equation*}
\begin{split}
|\cQ|w([n])+\sum_{i=1}^{\ell_2-1} |\cQ|\binom{\ell_2}{i}^{-1}(s-1)\binom{n}{n-i} & \le
|\cQ|+|\cQ|(s-1)\binom{n}{\lfloor \frac{n}{2} \rfloor}\sum_{i=1}^{\ell_2-1} \binom{\ell_2}{i}^{-1}\\
& =
|\cQ|\binom{n}{\lfloor \frac{n}{2} \rfloor}O_s\left(\frac{1}{n}\right).
\end{split}
\end{equation*}

There are no $s$ subsets of $A_1$ of equal size in $\cF$, since these would form a rank-preserving copy of $Y'_{h,s}$ together with the sets $A_1, A_2, \dots A_{h-1}$ and some set $Q\in\cF$, $A_{h-1}\subset Q$. (Such a set exists, since $A_{h-1}\in\cG$.)

A subset of $A_1$ of size $i$ appears in at most $\binom{\ell_1}{i}^{-1}\ell_1!N_2\ell_2!$ chains of $\cQ$. Therefore the total contribution to the sum (\ref{somesumhere}) by subsets of $A_1$ is at most
\begin{equation}\label{egyenlet}
\begin{split}
\ell_1!N_2\ell_2!w(\emptyset)+\sum_{i=1}^{\ell_1-1} \binom{\ell_1}{i}^{-1}\ell_1!N_2\ell_2!(s-1)\binom{n}{i} & \le
\ell_1!N_2\ell_2!+\ell_1!N_2\ell_2!(s-1)\binom{n}{\lfloor \frac{n}{2} \rfloor}\sum_{i=1}^{\ell_1-1} \binom{\ell_1}{i}^{-1}\\
& =
\ell_1!N_2\ell_2!\binom{n}{\lfloor \frac{n}{2} \rfloor}O_s\left(\frac{1}{n}\right).
\end{split}
\end{equation}

We will show that if $n$ is large and $\ell_1\ge \frac{n}{2}-2\sqrt{n\log n}$ then most chains between $\emptyset$ and $A_1$ avoid the elements of $\cG$, therefore $N_1$ is close to $\ell_1!$. There are at most $s-1$ sets of $\cG$ on any level (otherwise a rank-preserving copy of $Y'_{h,s}$ would be formed), and $\emptyset\not\in\cG$. There are $\ell_1!\binom{\ell_1}{i}^{-1}$ chains between $\emptyset$ and $A_1$ containing a set of size $i$. Therefore
$$\ell_1!-N_1\le (s-1)\sum_{i=1}^{\ell_1-1}\ell_1!\binom{\ell_1}{i}^{-1}=\ell_1! O\left(\frac{1}{n}\right).$$
This means that for large enough $n$, we have $\ell_1!\le 2N_1$.
Then (\ref{egyenlet}) can be continued as
$$\ell_1!N_2\ell_2!\binom{n}{\lfloor \frac{n}{2} \rfloor}O_s\left(\frac{1}{n}\right)\le 2N_1N_2\ell_2!\binom{n}{\lfloor \frac{n}{2} \rfloor}O_s\left(\frac{1}{n}\right)=|\cQ|\binom{n}{\lfloor \frac{n}{2} \rfloor}O_s\left(\frac{1}{n}\right).$$

To summarize, we found that the contribution to the sum $(\ref{somesumhere})$ from the subsets of $A_1$ and the supersets of $A_{h-1}$ is at most
$$|\cQ|\binom{n}{\lfloor \frac{n}{2} \rfloor}O_s\left(\frac{1}{n}\right).$$
For large enough $n$ this is smaller than $|\cQ|\left(\Sigma(n,h)-\Sigma(n,h-1)\right)$, which means that
$$\sum_{\cC\in \cQ}w(\cC)\le |\cQ|\Sigma(n,h).$$
This completes the proof.
\qed

\begin{remark}
We had to use a weighting technique in the above proof because the usual Lubell method (proving that $\sum_{F\in\cF}\binom{n}{|F|}^{-1}\le h$, and deducing $|\cF|\le\Sigma(n,h)$ from that) does not work for this problem. To see this, let $h\ge 3$, $n\ge 2h$ and consider the following set system:
$$\cF=\{F\in [n] ~|~ |F|\le h-2~{\rm or}~|F|\ge n-h+2\}.$$
For $s\ge 2^{h-2}$ this set system is $Y_{h,s}$-free and $Y'_{h,s}$-free (even in the original sense, not necessarily in the rank-preserving sense). However, we have $\sum_{F\in\cF}\binom{n}{|F|}^{-1}= 2(h-1)>h$.
\end{remark}

\bibliographystyle{acm}
\bibliography{biblio}
\end{document}